\documentclass[11pt]{amsart}
\usepackage{amssymb}

\newcommand{\dbar}{\ensuremath{\overline\partial}}

\newcommand{\C}{\ensuremath{\mathbb{C}}}

\newcommand{\D}{\ensuremath{\mathbb{D}}}

\makeatletter
\newcommand{\sumprime}{\if@display\sideset{}{'}\sum%
            \else\sum'\fi}
\makeatother

\begin{document}

\numberwithin{equation}{section}

\newtheorem{theorem}{Theorem}[section]
\newtheorem{proposition}[theorem]{Proposition}
\newtheorem{conjecture}[theorem]{Conjecture}
\def\theconjecture{\unskip}
\newtheorem{corollary}[theorem]{Corollary}
\newtheorem{lemma}[theorem]{Lemma}
\newtheorem{observation}[theorem]{Observation}
\newtheorem{definition}{Definition}
\numberwithin{definition}{section} 
\newtheorem{remark}{Remark}
\def\theremark{\unskip}
\newtheorem{question}{Question}
\def\thequestion{\unskip}
\newtheorem{example}{Example}
\def\theexample{\unskip}
\newtheorem{problem}{Problem}

\def\vvv{\ensuremath{\mid\!\mid\!\mid}}
\def\intprod{\mathbin{\lr54}}
\def\reals{{\mathbb R}}
\def\integers{{\mathbb Z}}
\def\N{{\mathbb N}}
\def\complex{{\mathbb C}\/}
\def\dist{\operatorname{dist}\,}
\def\spec{\operatorname{spec}\,}
\def\interior{\operatorname{int}\,}
\def\trace{\operatorname{tr}\,}
\def\cl{\operatorname{cl}\,}
\def\essspec{\operatorname{esspec}\,}
\def\range{\operatorname{\mathcal R}\,}
\def\kernel{\operatorname{\mathcal N}\,}
\def\dom{\operatorname{Dom}\,}
\def\linearspan{\operatorname{span}\,}
\def\lip{\operatorname{Lip}\,}
\def\sgn{\operatorname{sgn}\,}
\def\Z{ {\mathbb Z} }
\def\e{\varepsilon}
\def\p{\partial}
\def\rp{{ ^{-1} }}
\def\Re{\operatorname{Re\,} }
\def\Im{\operatorname{Im\,} }
\def\dbarb{\bar\partial_b}
\def\eps{\varepsilon}
\def\O{\Omega}
\def\Lip{\operatorname{Lip\,}}

\def\Hs{{\mathcal H}}
\def\E{{\mathcal E}}
\def\scriptu{{\mathcal U}}
\def\scriptr{{\mathcal R}}
\def\scripta{{\mathcal A}}
\def\scriptc{{\mathcal C}}
\def\scriptd{{\mathcal D}}
\def\scripti{{\mathcal I}}
\def\scriptk{{\mathcal K}}
\def\scripth{{\mathcal H}}
\def\scriptm{{\mathcal M}}
\def\scriptn{{\mathcal N}}
\def\scripte{{\mathcal E}}
\def\scriptt{{\mathcal T}}
\def\scriptr{{\mathcal R}}
\def\scripts{{\mathcal S}}
\def\scriptb{{\mathcal B}}
\def\scriptf{{\mathcal F}}
\def\scriptg{{\mathcal G}}
\def\scriptl{{\mathcal L}}
\def\scripto{{\mathfrak o}}
\def\scriptv{{\mathcal V}}
\def\frakg{{\mathfrak g}}
\def\frakG{{\mathfrak G}}

\def\ov{\overline}

\thanks{Research supported in part by NSF grants DMS-0805852 and DMS-1101678, and a U.S.-China Collaboration in Mathematical Research supplementary to NSF grant DMS-0500909,  and by Fok Ying Tung Education Foundation grant No.~111004 and Chinese NSF grant No.~11031008.}

\address{Department of Applied Mathematics, Tongji University, Shanghai, 200092, China}
\address{Department of Mathematical Sciences,
Rutgers University, Camden, NJ 08102, U.S.A.} \email{boychen@tongji.edu.cn} \email{sfu@camden.rutgers.edu}

\title{Stability of the Bergman kernel on a tower of coverings}
\author{Bo-Yong Chen and Siqi Fu}
\date{}
\maketitle

\bigskip

\leftline{\it Dedicated to Takeo Ohsawa Sensei}

\bigskip

\begin{abstract} We obtain several results about stability of the Bergman kernel on a tower of coverings on complex manifolds. An effective version of Rhodes' result is given for a tower of coverings on a compact Riemann surface
of genus $\ge 2$. Stability of the Bergman kernel is established for towers of coverings on hyperbolic Riemann surfaces and on complete K\"{a}hler manifolds satisfying certain potential conditions. As a consequence, stability  of the Bergman kernel is established for any tower of coverings of Riemann surfaces when the top manifold is simply-connected.
\bigskip


\smallskip

\end{abstract}

\section{Introduction}

The classical Bergman kernel--the reproducing kernel for $L^2$-holomorphic functions--has long played an
important role in complex analysis. Its generalization to complex manifolds--in this case, the kernel for
the projection onto the space of harmonic $(p, q)$-forms with $L^2$-coefficients--is encoded with
information on the algebraic and geometric structures of the underlying manifolds. How the Bergman kernel
behaves as the underlying structures changes is a problem that has been extensively studied in a number of
settings. Convergence of the Bergman kernel associated to tensor powers of a positive holomorphic line
bundle over a compact complex manifold as the power goes to infinity was established in the celebrated work of Tian \cite{Tian90}, Zelditch \cite{Zelditch98}, and Catlin \cite{Catlin99}. (See  \cite{BBS08} and
references therein for recent developments.)

In this paper, we study stability of the Bergman kernel on a quotient $\widetilde M/\Gamma$ of a complex manifold $\widetilde M$ by a free and properly discontinuous group $\Gamma$ of automorphisms of $\widetilde M$ as $\Gamma$ shrinks to the identity.  We first recall the setup in Riemannian geometry. Let $\widetilde M$ be a Riemannian manifold and $\Gamma$ a free and properly discontinuous group of isometries of $\widetilde M$. A {\it tower of subgroups} of $\Gamma$ is a nested
sequence of subgroups $\Gamma=\Gamma_1\supset\Gamma_2\supset\cdots\supset\Gamma_j \supset\cdots\supset \cap
\Gamma_j=\{ \text{id}\}$ such that $\Gamma_j$ is a normal subgroup of $\Gamma$ of finite index
$[\Gamma:\Gamma_j]$ for each $j$ (see \cite[p.~135]{DeGeorgeWallach78}). The smooth manifolds
$M_j=\widetilde M/\Gamma_j$ are equipped with the push-downs of the Riemannian metric on $\widetilde M$.
The family $\{ M_j \}$ is called a {\it tower of coverings} on the Riemannian manifold $M=\widetilde
M/\Gamma$. We will also refer to $\widetilde M$ as the {\it top manifold}, $M_j$ the {\it covering
manifolds},  and $M=\widetilde M/\Gamma$ the {\it base manifold} of the tower of coverings. (In
applications, the top manifold is usually assumed to be the universal covering of the base manifold.
However, we make no such assumption in this paper.) It is well known that every Riemannian manifold whose
fundamental group is isomorphic to a finitely generated subgroup of $SL(n,{\mathbb C})$ admits a tower of
coverings with the top manifold being the universal covering (cf. \cite[Theorem~B and
Proposition~2.3]{Borel63}). This is the case, for instance, for an arithmetic quotient of a bounded symmetric domain. Limiting behavior of the spectrum of the Laplacian on $M_j$ as $j\to\infty$ was studied
for a tower of coverings on symmetric spaces of non-compact type by DeGeorge-Wallach \cite{DeGeorgeWallach78}.

In his work \cite{Kazhdan70, Kazhdan83}, Kazhdan employed the Bergman kernel to study arithmetic varieties
and initiated the study of the Bergman kernel on a tower of coverings on a complex manifold. It follows
from his work that for a tower of coverings $\{\widetilde{M}/\Gamma_j\}$ on a compact complex manifold, the Bergman kernel on the universal covering $\widetilde{M}$ is nontrivial provided $\lim\sup_{j\rightarrow \infty} h^{n, 0}(M_j)/[\Gamma:\Gamma_j]>0$, where $h^{n, 0}(M_j)$ is the dimension of the space of global sections of the canonical line bundle on $M_j$ (see \cite[Theorem~1]{Kazhdan83}). Kazhdan suggested that for a tower of coverings on a Riemann surface, the pull-backs of the Bergman metric on $M_j$ converges to that of the
upper half plane $\widetilde M={\mathbb H}$ (see \cite[p.~12]{Mumford75}). In \cite[p.~139]{Yau86}, Yau
stated, as a result attributed to Kazhdan, that this also holds for a tower of coverings on any complex
manifold. However, no proof of this statement has been published as far as we know. (See
Section~\ref{sec:compact} below for a discussion on the link between Kazhdan's inequality and convergence
of the Bergman kernels.)

For brevity, a tower of coverings $\{M_j\}$ on a complex manifold is said to be {\it Bergman stable}
if the pull-back of the Bergman kernel on $M_j$ converges locally uniformly to that of the top manifold
$\widetilde M$ as $j\to\infty$. In 1993, Rhodes showed that a tower of covering on a compact
Riemann surface of genus $g\ge 2$ is indeed Bergman stable~\cite{Rhodes93}. Donnelly
\cite{Donnelly96} proved analogous results for a tower of coverings on a Riemannian manifold $M$ under the
conditions that $\widetilde M$ has bounded sectional curvature and the smallest nonzero eigenvalue of the
Laplacian on $M_j$ is uniformly bounded from below by a positive constant. His method was based on
Cheeger-Gromov-Taylor's estimates of the heat kernel \cite{CGT82} and Atiyah's $L^2$-index theorem
\cite{Atiyah76}.  Building on Donnelly's work, Yeung showed in \cite{Yeung00a, Yeung00b}  that the
canonical line bundle of $M_j$ is very ample when the injectivity radius of $M_j$ is greater than a certain
effective constant depending on the top manifold $\widetilde{M}$ and the base manifold $M$. More recently,
using Donnelly-Fefferman's $L^2$-estimate for the $\bar\partial$-operator, Ohsawa \cite{Ohsawa09} established
Bergman stability for a tower of coverings on a complex manifold under certain assumptions on successive
approximations of $\dbar$-closed $(n, 0)$-forms on $M_j$ by those on $\widetilde M$. He further gave an example of a tower of {\it branched} coverings on a compact Riemann surface that is not Bergman stable (\cite{Ohsawa10}).

For a tower of coverings on a compact Riemann surface, we establish the following effective version of
Rhodes' theorem:

\begin{theorem}\label{th:eff} Let $M_j=\D/\Gamma_j$ be a tower of
coverings on a compact Riemann surface of genus $g\ge 2$. Let $\tau_j$ be the injectivity radius of $M_j$
and $|\cdot|_{\rm hyp}$ the pointwise length with respect to the hyperbolic metric. Then the Bergman kernel
$K_{M_j}$ of $M_j$ satisfies
\begin{equation}\label{eq:eff}
|4\pi|K_{M_j}|_{\rm hyp}-1|\le \frac{12\cdot 3^{2/3}}{\pi}(g-1)^{1/3}e^{-\tau_j/3}
\end{equation}
when $\tau_j\ge \log 3$. Furthermore, a similar estimate also holds for the Bergman metric.
\end{theorem}

Our proof of the above theorem is elementary; it uses only the Gauss-Bonnet formula and the reproducing
property of the Bergman kernel. For a tower of coverings on a compact complex manifold, we exhibit a
connection between Bergman stability and the theory of $L^2$-Betti  numbers, an area studied extensively in the literatures (cf. \cite{CheegerGromov85a, CheegerGromov85b, Luck94, Yeung94}; see Section~\ref{sec:compact} below).

The main focus in this paper, however, is on towers of coverings on {\it noncompact} complex manifolds. Recall that a Riemann surface $M$ is {\it hyperbolic} if it carries a negative nonconstant subharmonic
function. This is equivalent to existence of the Green function on $M$. (We refer the reader to
\cite[Chapter IV]{FarkasKra80} for relevant background material.) For noncompact Riemann surfaces, as a simple application of the classical Myrberg's formula \cite{Myrberg33} and an idea from \cite{Fu01}, we have:

\begin{theorem}\label{th:hyp}
Any tower of coverings on a hyperbolic Riemann surface is Bergman stable.
\end{theorem}

Our main result on higher dimensional non-compact complex manifolds
can be stated as follows:

\begin{theorem}\label{th:g}  Let $M$ and $\widetilde M$ be complete K\"{a}hler manifolds with associated
K\"{a}hler forms $\omega$ and $\widetilde\omega$ respectively. Let $M_j=\widetilde M/\Gamma_j$ be a tower
of coverings on $M$.  Then the tower is Bergman stable provided the following two conditions are satisfied:
\begin{enumerate}
\item There exist a compact set $K\subset M$, a $C^2$-smooth plurisubharmonic
function $\psi$ on $M\setminus K$, and a constant $C>0$ such that $C^{-1}\omega\le
\partial\bar\partial\psi\le C\omega$ and $\partial\bar\partial\psi\ge
C^{-1}\partial\psi\wedge\bar\partial\psi$ on $M\setminus K$.

\item There exist a $C^2$-smooth plurisubharmonic
function $\widetilde\psi$ on $M$ and a constant $C>0$ such that $C^{-1}\widetilde\omega\le
\partial\bar\partial\widetilde\psi\le C\widetilde\omega$ and $\partial\bar\partial\widetilde\psi\ge
C^{-1}\partial\widetilde\psi\wedge\bar\partial\psi$ on $\widetilde M$.
\end{enumerate}
\end{theorem}

It is easy to see that the above conditions are satisfied if the base manifold $M$ is a hyperconvex complex
manifold--namely if it admits a $C^2$ strictly plurisubharmonic proper map $\rho: M\rightarrow
[-1,0)$. As a consequence of Theorem~\ref{th:hyp} and Theorem~\ref{th:g}, we have:

\begin{theorem}\label{th:r}
Any tower of coverings of Riemann surfaces with the top manifold simply-connected is Bergman stable.
\end{theorem}

Our proof of Theorem~\ref{th:g} uses Donnelly-Fefferman type $L^2$-estimates \cite{DonnellyFefferman83} for
the $\dbar$-Laplacian.  It also uses spectral theory: the conditions on $\widetilde M$ and $M$ ensure that
the spectrum and essential spectrum of the respective complex Laplacian on $(n, 1)$-forms on $\widetilde M$ and $M$ are positive. However, here instead of estimating the heat kernel as in \cite{Donnelly96}, we study the spectral (Bergman) kernel. This enables us to streamline the arguments and replace curvature
conditions by potential theoretic conditions on manifolds $M$ and $\widetilde M$.

This paper is organized as follows. In Section~\ref{sec:prelim}, we review basic definitions and properties
of towers of coverings and the Bergman kernel. Theorem~\ref{th:eff} is proved in Section~\ref{sec:eff}.
Section~\ref{sec:compact} contains a discussion on the connection between Kazhdan's inequality and stability
of the Bergman kernel on towers of coverings on compact complex manifolds. Theorem~\ref{th:hyp} is proved in
Section~\ref{sec:hyp} and Theorem~\ref{th:g} in Section~\ref{sec:g}. Applications of Theorem~\ref{th:g} to
quotients of the ball and polydisc are given in Section~\ref{sec:app}.

Throughout the paper, we will use $C$ to denote a positive constant which may be different in different
appearances. We will also use $f\gtrsim g$ to denote $f\ge Cg$ where $C$ is a constant, its independence of
certain parameters being clear from the contexts, and use $A\approx B$ to denote $A\gtrsim B$ and $B\gtrsim
A$.

\section{Preliminaries}\label{sec:prelim}

We briefly recall the definition of the Bergman kernel and metric on a complex manifold (see
\cite{Kobayashi59}).  Let $X$ be a complex manifold and $A^2_{(n, 0)}(X)$ space of square integrable
holomorphic $(n, 0)$-forms $f$ equipped with the inner product
\begin{equation}\label{eq:l2}
\langle f, g\rangle=i^{n^2} 2^{-n} \int_X f\wedge \overline{g}.
\end{equation}
The Bergman kernel is an $2n$-form on $X\times X$ given by
\begin{equation}\label{eq:berg}
K_X(z, w)=\sum_{j=1}^\infty b_j(z)\wedge \overline{b_j(w)}
\end{equation}
where $\{b_j\}$ is an orthonormal basis for $A^2_{(n, 0)}(X)$.  Write $K^z_X(\cdot)=K_X(z, \cdot)$. Then the Bergman kernel has the following reproducing property:
\begin{equation}\label{eq:reproducing}
f(z)=(-1)^{n^2}\int_X K^z_X\wedge f, \qquad \forall f\in A^2_{(n, 0)}(X).
\end{equation}
The Bergman kernel on diagonal
$K_X(z)=K_X(z, z)$ is a biholomorphically invariant $(n, n)$-form on $X$ with the extremal property:
\begin{equation}\label{eq:extreme}
K_X(z)=\max\{f(z)\wedge \overline{f(z)} \mid f\in A^2_{(n, 0)}(X), \|f\|=1\}
\end{equation}
and the decreasing property: $K_X(z)\le K_Y(z)$ if $Y$ is a subdomain of $X$.  Given a local holomorphic
coordinate chart $(z_1, \ldots, z_n)$, write $\omega_n=\wedge_{j=1}^n (\frac{i}2 dz_j\wedge d\bar z_j)$
and
\[
K_X(z)=K^*(z) \omega_n.
\]
When $K^*(z)>0$, the Bergman (pseudo-)metric is given by
\[
ds^2_X=\sum_{j, k=1}^n \frac{\partial^2\log K^*(z)}{\partial z_j\partial\bar z_k} dz_j d\bar z_k.
\]
The Bergman metric is a bihilomorphically invariant metric and it can be regarded as the pull-back of the
Fubini-Study metric of (possibly infinitely dimensional) complex projective spaces (\cite{Kobayashi59}).

We review elements of the $L^2$-cohomology theory for the $\bar{\partial}-$operator. Let $(M,\omega)$
be a complex hermitian manifold of complex dimension $n$. Let $C_0^{p,q}(M)$ be the space of $C^\infty$
$(p,q)-$forms with compact supports on $M$ and let $L^{p,q}_{(2)}(M)$ be the completion of $C^{p,q}_0(M)$
with respect to the following $L^2-$norm
$$
\|u\|=\left(\int_M |u|^2dV\right)^{1/2}
$$
where $|\cdot|$ is the point-wise norm corresponding to $\omega$ and $dV$ the volume form. The weak
maximal extension $\bar{\partial}:L^{p,q}_{(2)}(M)\rightarrow L^{p,q+1}_{(2)}(M)$ is a densely
defined closed operator. Let $\bar{\partial}^\ast$ be the adjoint of $\bar{\partial}$. Then the
$\bar{\partial}$-Laplacian is $\Box=\bar{\partial}\bar{\partial}^\ast+\bar{\partial}^\ast\bar{\partial}$
and the space of $L^2$-harmonic $(p,q)-$forms is given by
$$
{\mathcal H}^{p,q}_{(2)}(M)=\left\{u\in L^{p,q}_{(2)}(M):\Box u=0\right\}=\left\{u\in
L^{p,q}_{(2)}(M):\bar{\partial}u=\bar{\partial}^\ast u=0\right\}.
$$

We now review relevant basic facts on covering spaces. Let $(\widetilde{M}, \widetilde{\omega})$ be a
Riemannian manifold. Let $\Gamma$ be a subgroup of the isometrics that acts freely and properly
discontinuously on $\widetilde{M}$. (Recall that $\Gamma$ acts {\it freely} if the identify map is the only
element in $\Gamma$ that has a fixed point and {\it properly discontinuously} if for any compact set $K$,
there is only finitely many $\gamma\in\Gamma$ such that $K\cap \gamma K\not=\emptyset$.) Let
$M=\widetilde{M}/\Gamma$ be the quotient manifold and $\pi\colon \widetilde M\to M$ the covering map. We
equip $M$ with the push-down metric $\omega$ from $\widetilde{M}$ so that
$\pi^*(\omega)=\widetilde{\omega}$.   Denote by $d_{\widetilde M}$ and $d_M$ the distances on
$\widetilde{M}$ and $M$ associated with $\widetilde{\omega}$ and $\omega$ respectively. For
$x\in\widetilde{M}$, let
\begin{equation}\label{eq:dirichlet}
D(x)=\{y\in\widetilde{M} \mid d_{\widetilde M}(y, x) < d_{\widetilde M}(y, \gamma x),
\forall\gamma\in\Gamma\setminus\{1\}\}
\end{equation}
be the {\it Dirichlet fundamental domain} with center at $x$.  It is easy to see that no pair of points in
$D(x)$ are equivalent under $\Gamma$ and every point in $\widetilde{M}$ has an equivalence in $D(x)$ or its
boundary. Let
\begin{equation}\label{eq:tau}
\tau(x)=\frac12\inf\left\{d_{\widetilde{M}}(x,\gamma x):\gamma\in \Gamma\setminus\{1\}\right\}.
\end{equation}
Evidently, the geodesic ball $B(x, \tau(x))$ is contained in $D(x)$. Moreover, when $\widetilde{M}$ has not
conjugate points (i.e., any two points are joint uniquely--up to reparametrization--by a geodesic),
$\tau(x)$ is the injectivity radius of $\pi(x)$ in $M$. In particular, this is the case when $\widetilde
M=\D$, the unit disc.

Let $\{\Gamma_j\}$ be a tower of subgroups of $\Gamma$. Denote by $\tau_j(x)$ the quantity defined by
\eqref{eq:tau} with $\Gamma$ replaced by $\Gamma_j$. Since $\tau_j(\cdot)$ is invariant under $\Gamma_j$, it
can be pushed down onto $M_j$.  The following lemma is well known (compare, e.g.,
\cite[Theorem 2.1]{DeGeorgeWallach78} and \cite[Lemma~2.1]{Donnelly83}):

\begin{lemma}\label{lm:inj} $\tau_j(x)$ is an increasing sequence of positive continuous functions such
that $\tau_j(x)\rightarrow \infty$ locally uniformly on $\widetilde{M}$  as $j\rightarrow \infty$.
\end{lemma}

\begin{proof} For the reader's convenience, we include a proof here. Since $\tau_j(x)$ is the infinum of a
sequence of continuous functions, it is itself upper semi-continuous. To prove that $\tau_j(x)$ is
continuous, it suffices to show that $A_\alpha=\{x\in\widetilde M \mid \tau_j(x)>\alpha\}$ is open for any
$\alpha\ge 0$. Let $x_0\in A_\alpha$. Choose $\eps>0$ sufficiently small so that $\tau_j(x_0)>\alpha+\eps$.
Then for any $x\in B(x_0, \eps/2)$ and $\gamma\in\Gamma_j$,
\[
\begin{aligned}
d_{\widetilde M}(x, \gamma x)&\ge d_{\widetilde M}(x_0, \gamma x_0)-d_{\widetilde M}(x, x_0)-d_{\widetilde
M}(\gamma x, \gamma x_0)\\ &=d_{\widetilde M}(x_0, \gamma x_0)-2d_{\widetilde M}(x, x_0)\ge 2\alpha+\eps.
\end{aligned}
\]
Thus $B(x_0, \eps/2)\subset A_\alpha$. Therefore, $A_\alpha$ is open and hence $\tau_j(x)$ is continuous.
It follows from the proper discontinuous property of $\Gamma$ that $\tau_j(x)>0$ and $\tau_j(x)\to\infty$
for any $x\in\widetilde M$. Since $\Gamma_j\supset\Gamma_{j+1}$, $\tau_j(x)\le\tau_{j+1}(x)$. It then
follows from Dini's theorem that $\tau_j(x)\to\infty$ locally uniformly on $\widetilde M$. \end{proof}

Hereafter, we assume that $(\widetilde{M}, \widetilde{\omega})$ is a complex Hermitian manifolds and
$\Gamma\subset {\rm Aut(\widetilde{M})}$,  the automorphism group of $\widetilde{M}$. Let $M_j=\widetilde
M/\Gamma_j$ and $\widetilde{p}_j\colon \widetilde{M}\to M_j$ be the natural projection. Throughout the
paper, when it is contextually clear, we will identify $K_{\widetilde M}$ with its push-down to $M$ and
likewise $K_{M_j}$ with its pull-back on $\widetilde M$ for the economy of notations. The following
proposition establishes the upper semi-continuity of the Bergman kernels $K_{M_j}$ on a tower of coverings on
complex manifolds (compare \cite[Proposition 1.2]{Donnelly96}).

\begin{proposition}\label{prop:upper}
For each $z\in \widetilde{M}$, ${\lim\sup}_{j\rightarrow \infty} \widetilde{p}_j^*(K_{M_j})(z)\le
K_{\widetilde{M}}(z)$.
\end{proposition}

\begin{proof}  Let $D_j(z)\subset\widetilde M$ be the Dirichlet fundamental domain of $M_j$ as defined in
\eqref{eq:dirichlet} by replacing $\Gamma$ by $\Gamma_j$.  Then $\widetilde{p}_j$ maps $D_j(z)$
biholomorphically onto its image $\widetilde{p}_j(D_j(z))$.  It follows from the decreasing property of the
Bergman kernel that
\[
K_{M_j}(\widetilde{p}_j(z))\le K_{\widetilde{p}_j(D_j(z))}(\widetilde{p}_j(z)).
\]
Since $B(z, \tau_j(z))\subset D_j(z)$, we have
\begin{equation}\label{eq:rama}
\widetilde{p}^*_j(K_{M_j})(z)\le \widetilde{p}^*_j(K_{\widetilde{p}_j(D_j(z))})(z)=K_{D_j(z)}(z)\le K_{B(z,
\tau_j(z))}(z).
\end{equation}
We then conclude the proof by applying Lemma~\ref{lm:inj} and Ramadanov's theorem (\cite{Ramadanov67}).
\end{proof}

As an application of the above proposition, we provide a proof of the following version of Kazhdan's
inequality (see \cite[Theorem 1 and its proof]{Kazhdan83}; also \cite[pp.~13 and pp.~153]{Gromov93}):

\begin{proposition}\label{prop:kaz}  When $M$ is compact,
\begin{equation}\label{eq:kaz}
\limsup_{j\to\infty}\frac{h^{n, 0}(M_j)}{[\Gamma:\Gamma_j]}\le \int_{M} K_{\widetilde M}.
\end{equation}
\end{proposition}

\begin{proof} Since $M$ is compact, there exists a compact set $A\subset\widetilde M$ such that
$\widetilde{p}(A)=M$, where $\widetilde{p}=\widetilde{p}_1$ is as before the natural projection from
$\widetilde M$ onto $M=M_1$.  Let $\eps$ be the minimum of $\tau(x)$ on $A$.  We cover $A$ by finitely many
geodesic balls $\{B(z_k, r_k/2)\}_{k=1}^m$ with $z_k\in A$ and $r_k<\eps$ such that each $B(z_k, r_k)$ is
contained in a normal neighborhood in $\widetilde M$. It follows from \eqref{eq:rama} that for $z\in B(z_k,
r_k/2)$,
\[
\widetilde{p}^*_j(K_{M_j})(z)\le K_{B(z, \tau(z))}\le K_{B(z_k, r_k/2)}(z).
\]
It then follows that $K_{M_j}$ is uniformly bounded from above on $M$. (Here we identify $K_{M_j}$ with its
push-down onto $M$.)  Inequality \eqref{eq:kaz} is then obtained by integrating both sides of the
inequality in Proposition~\ref{prop:upper} over $M$, the dominated convergence theorem, and the fact that
\[
\int_{M} K_{M_j}=\frac{1}{[\Gamma:\Gamma_j]}\int_{M_j} K_{M_j}=\frac{h^{n, 0}(M_j)}{[\Gamma:\Gamma_j]}.
\]
\end{proof}

The following proposition establishes the link between the convergence of the Bergman kernel and the Bergman metric
(compare \cite{Ramadanov67}).

\begin{proposition}\label{prop:km} Suppose the Bergman kernel $K_{\widetilde M}(z, z)$ is positive and the
tower of coverings $M_j$ is Bergman stable. Then the pull-back of Bergman metric of $M_j$ converges
locally uniformly to the Bergman metric on $\widetilde{M}$.
\end{proposition}

\begin{proof} Let $z_0\in \widetilde M$. Recall that $B_j=B(z_0, \tau_j(z_0))\subset D_j(z_0)$, the Dirichlet
fundamental domain of $M_j$ with center at $z_0$.  Denote by $\|\cdot \|_{B_j}$ the $L^2$-norm as defined
by \eqref{eq:l2} over $B_j$. Let $w\in B_j$. It follows from the reproducing property of the Bergman kernel
that
\begin{eqnarray*}
&&\left\|\widetilde{p}^*_j(K_{M_j})(\cdot, w)-K_{B_j}(\cdot, w)\right\|^2_{B_j}\\
&&=\left\|\widetilde{p}_j^*(K_{M_j})(\cdot, w)\right\|^2_{B_j}-2{\rm Re}\langle
\widetilde{p}_j^*(K_{M_j})(\cdot, w), K_{B_j}(\cdot, w)\rangle+\left\|K_{B_j}(\cdot, w)\right\|^2_{B_j}\\
&&=\left\|K_{M_j}(\cdot, w)\right\|^2_{\widetilde{p}_j(B_j)}-2\widetilde{p}_j^*(K_{M_j})(w, w)+K_{B_j}(w,
w)\\ &&\le \left\|K_{M_j}(\cdot, w)\right\|^2_{M_j}-2\widetilde{p}_j^*(K_{M_j})(w, w)+K_{B_j}(w, w)\\
&&=K_{B_j}(w, w)-\widetilde{p}_j^*(K_{M_j})(w, w).
\end{eqnarray*}
Similarly,
\[
\|K_{\widetilde M}(\cdot, w)-K_{B_j}(\cdot, w)\|^2_{B_j}\le K_{B_j}(w, w)-K_{\widetilde M}(w, w).
\]
Therefore,
\begin{equation}\label{eq:km}
\|K_{\widetilde M}(\cdot, w)-\widetilde{p}_j(K_{M_j})(\cdot, w)\|^2_{B_j}\le 2\big(2K_{B_j}(w,
w)-K_{\widetilde M}(w, w)-\widetilde{p}_j(K_{M_j})(w, w)\big).
\end{equation}
It follows from Ramadanov's theorem \cite{Ramadanov67} and the Bergman stability assumption that the right
hand side above converges locally uniformly to zero for $w$ near $z_0$. Since the Bergman metric is given
locally by
\[
\partial\bar\partial\log K^*=\frac{\partial\bar\partial K^*}{K^*}-\frac{\partial K^*\wedge\bar\partial
K^*}{(K^*)^2},
\]
we then conclude the proof of the proposition by applying the Cauchy estimate.
\end{proof}

\section{Effective estimates}\label{sec:eff}

Recall that the hyperbolic metric on the unit disc $\D$ is given by
$$
ds^2_{\rm hyp}=\frac{4\, |dz|^2}{(1-|z|^2)^2}
$$
and the hyperbolic distance between $z$ and $0$ is
$$
{\rm dist}_{\rm hyp}(0,z)=\log\frac{1+|z|}{1-|z|}.
$$
It follows that the Euclidean and hyperbolic balls $B_{\rm eucl}(0,r)$ and $B_{\rm hyp}(0,\tau)$  are
identical if and only if
$$
r=\frac{e^{\tau}-1}{e^{\tau}+1}.
$$
Furthermore,
$$
K_{B_{\rm hyp}(0,\tau)}(0)=\frac1\pi\frac{(e^{\tau}+1)^2}{(e^{\tau}-1)^2}dz\wedge d\bar{z}
$$
where $K_{B_{\rm hyp}(0,\tau)}$ denotes the Bergman kernel form on $B_{\rm hyp}(0,\tau)$.  Let $p\colon
\D\to M$ be a covering map on a Riemann surface $M$.  Then the hyperbolic metric $ds^2_{{\rm hyp}, M}$
satisfies
\[
p^*(ds^2_{{\rm hyp}, M})=ds^2_{{\rm hyp}}.
\]
Thus for a form $K$ on $M$, we have $|p^*(K)(z)|_{{\rm hyp}}=|K(p(z))|_{{\rm hyp}, M}$, where $|\cdot
|_{{\rm hyp}}$ denotes the pointwise norm with respect to the hyperbolic metric. (We will drop subscript
$M$ when doing so causes no confusion.)

We now prove Theorem~\ref{th:eff}. Let $\widetilde{p}_j\colon \D\to M_j=\D/\Gamma_j$ and $p_j\colon M_j\to
M=M_j/(\Gamma/\Gamma_j)$ be the natural projections. Let $\tau_j(w)$ be the injectivity radius at $w\in
M_j$ and let $\tau_j$ denote the injectivity radius of $M_j$. For any $z\in\D$, we have
$$
|K_{M_j}(\widetilde{p}_j(z))|_{\rm hyp}=|\widetilde{p}_j^*(K_{M_j})(z)|_{{\rm hyp}}\le  |K_{B_{\rm
hyp}(z,\tau_j(z))}(z)|_{\rm hyp}\le |K_{B_{\rm hyp}(0,\tau_j)}(0)|_{\rm hyp}.
$$
Hence
\begin{equation}\label{eq:eff1}
|K_{M_j}(\widetilde{p}_j(z))|_{\rm hyp}-\frac1{4\pi}\le \frac1{\pi}\frac{e^{\tau_j}}{(e^{\tau_j}-1)^2}.
\end{equation}
It follows from the Gauss-Bonnet theorem that
$$
4\pi(g_j-1)={\rm vol}_{\rm hyp}(M_j).
$$
Since ${\rm vol}_{\rm hyp}(M_j)=[\Gamma:\Gamma_j]{\rm vol}_{\rm hyp}(M)$, we obtain
$$
\frac{g_j}{[\Gamma:\Gamma_j]}=g-1+\frac1{[\Gamma:\Gamma_j]}.
$$
Note that $g_j=\int_{M_j}|K_{M_j}|_{\rm hyp}dV_{\rm hyp}$ and $|K_{M_j}|_{\rm hyp}$ is invariant under the
deck transformations of $p_j\colon M_j\rightarrow M$.  Thus
\begin{equation}\label{eq:eff2}
\int_{M} |K_{M_j}|_{\rm hyp}dV_{\rm hyp}=\int_M \frac1{4\pi}dV_{\rm hyp}+\frac1{[\Gamma:\Gamma_j]}\ge
\int_M \frac1{4\pi}\, dV_{\rm hyp}.
\end{equation}
Here we identify $K_{M_j}$ with its push-down to $M$.

It suffices to prove the theorem at $z=0$; the general case is reduced to this case by  applying a
M\"{o}bius transformation. Let $r_j=(e^{\tau_j}-1)/(e^{\tau_j}+1)$. Then the Euclidean disk $B_{\rm
eucl}(0, r_j)=B_{\rm hyp}(0, \tau_j)$, the hyperbolic disk, and it is contained in the fundamental domain
$D_j(0)$ of $M_j$. Write $\widetilde{p}_j^*(K_{M_j})(z)=K^{\ast}_{M_j}(z)dz\wedge d\bar{z}$ on $B_j=B_{\rm
eucl}(0, r_j)$. Let $\epsilon$ be a sufficiently small positive constant to be chosen. For $w\in B_j$, let
$f$ be a holomorphic 1-form on $M_j$ with unit $L^2-$norm (as defined by \eqref{eq:l2}) such that
$K_{M_j}(\widetilde p_j(w))=f_j(\widetilde p_j(w))\wedge \overline{f_j(\widetilde p_j(w))}$. Write
$\widetilde{p}_j^*(f)(z)=f^\ast(z)dz$. Then
$$
\int_{B_j}|f^\ast|^2 dV_{\rm eucl}=\frac12\left|\int_{B_j}\widetilde{p}_j^*(f)\wedge
\overline{\widetilde{p}_j^*(f)}\right|\le\frac{1}{2} \left|\int_{M_j}f\wedge \bar{f}\right|=1.
$$
Since
\begin{equation}\label{eq:cauchy}
(f^*(z))^2=\frac{r^2_j}{\pi}\int_{|\zeta|<r_j}\frac{(f^*(\zeta))^2}{(r^2_j-z\bar\zeta)^2} dV_{\rm eucl},
\quad z\in B_j,
\end{equation}
it follows that
$$
|(f^\ast(z))^2-(f^\ast(z'))^2|\le \frac{48}{\pi r^3_j} |z-z'|
$$
for all $z,z'\in \frac12 B_j$.  Now suppose $w\in \epsilon B_j$ where $\epsilon$ is a sufficiently small
number to be chosen. Then
\begin{equation}\label{eq:eff3}
|\widetilde{p}_j^*(K_{M_j})(w)|_{\rm hyp}=\frac{(1-|w|^2)^2}4|f^\ast(w)|^2\le
\frac14(|f^\ast(0)|^2+\frac{48\epsilon}{\pi r^2_j})\le |K_{M_j}(0)|_{\rm hyp}+\frac{12\epsilon}{\pi r^2_j}.
\end{equation}
From \eqref{eq:eff2} and then \eqref{eq:eff1}, we have
$$
\int_{\widetilde{p}_1(\epsilon B_j)}\left(\frac1{4\pi}-|K_{M_j}|_{\rm hyp}\right)dV_{\rm
hyp}\le\int_{M\backslash\widetilde{p}_1(\epsilon B_j)}\left(|K_{M_j}|_{\rm hyp}-\frac1{4\pi}\right)dV_{\rm
hyp}\le  \frac{4(g-1)e^{\tau_j}}{(e^{\tau_j}-1)^2}.
$$
Combining this with \eqref{eq:eff3}, we obtain
$$
\frac1{4\pi}-|K_{M_j}|_{\rm hyp}(0)- \frac{12\epsilon}{\pi r^2_j}\le
\frac{4(g-1)e^{\tau_j}}{(e^{\tau_j}-1)^2}\frac1{{\rm vol}_{\rm hyp}(\epsilon B_j)}.
$$
Since ${\rm vol}_{\rm hyp}(\epsilon B_j)=4\pi (\epsilon r_j)^2/(1-(\epsilon r_j)^2)\ge 4\pi \epsilon^2
r^2_j$,
\[
\frac1{4\pi}-|K_{M_j}|_{\rm hyp}(0)\le \frac{12\epsilon}{\pi r^2_j}+ \frac{(g-1)e^{\tau_j}}{\pi
(e^{\tau_j}-1)^2}\frac{1}{\epsilon^2 r_j^2}.
\]
Choosing $\epsilon=\left((g-1)e^{\tau_j}/6(e^{\tau_j}-1)^2\right)^{1/3}$, we then have
$$
\frac1{4\pi}-|K_{M_j}|_{\rm hyp}(0)\le \frac{18}{\pi r^2_j}\left(\frac{(g-1)
e^{\tau_j}}{6(e^{\tau_j}-1)^2}\right)^{1/3}.
$$
Note that the right hand side above is greater than that in \eqref{eq:eff1}. Since $\tau_j\ge \log 3$, we have
$$
e^{\tau_j}-1\ge 2,\ \ \ \ \ r_j=\frac{e^{\tau_j}-1}{e^{\tau_j}+1}\ge \frac12.
$$
Hence
\[
\left| |K_{M_j}(0)|_{\rm hyp}-\frac1{4\pi}\right|\le \frac{C}{\pi}(g-1)^{1/3}e^{-\tau_j/3}
\]
where
$$
C=18\cdot 4 \cdot (6\cdot 4)^{-1/3}=12\cdot 3^{2/3}.
$$
This concludes the proof of Theorem~\ref{th:eff} for the Bergman kernel.

We now show how to obtain effective estimates for the Bergman metric, without keeping track of the
numerical constants.  Let $K^*_{M_j}(z,w)$ denote the function on $\D$ representing the pull-back of the
Bergman kernel form on $M_j$.   Let $K^*_\D$ and $K^*_{B_j}$ be the Bergman kernel functions of $\D$ and
$B_j=B_{\rm eucl}(0, r_j)$ respectively. Assume that $\tau_j\ge \log 3$. Then $r_j\ge 1/2$.  From the first
part of the theorem, we know that for $w\in \frac12\D$,
\begin{equation}\label{eq:eff4}
|K^*_{M_j}(w, w)-K^*_\D(w, w)|\le C(g-1)^{1/3}e^{-\tau_j/3}.
\end{equation}
Furthermore, a simple calculation yields that
\begin{equation}\label{eq:eff5}
|K^*_{B_j}(w, w)-K^*_\D(w, w)|\le C e^{-\tau_j}.
\end{equation}
Following the same lines of argument as in the proof of \eqref{eq:km}, we have
\begin{equation}\label{eq:eff6}
\int_{B_j}|K_{M_j}^\ast(z,w)-K^\ast_{\D}(z,w)|^2 dV_{\rm eucl}(z)\le 2\big(2K^*_{B_j}(w, w)-K^*_{\D}(w,
w)-K^*_{M_j}(w, w)\big).
\end{equation}
Combining \eqref{eq:eff4}-\eqref{eq:eff6}, we then obtain
$$
\int_{\frac12 \D}|K_{M_j}^\ast(z,w)-K^\ast_{\D}(z,w)|^2 dV_{\rm eucl}(z)\le C(g-1)^{1/3}e^{-\tau_j/3}.
$$
Using the reproducing property of the Bergman kernel on $\frac12\D$ as in \eqref{eq:cauchy} and applying
the Cauchy-Schwarz inequality, we have for any integers $\alpha, \beta\ge 0$,
\[
\left|\left(\frac{\partial^{\alpha+\beta} K^*_{M_j}}{\partial z^\alpha\partial\bar
z^\beta}-\frac{\partial^{\alpha+\beta} K^*_{\D}}{\partial z^\alpha\partial\bar z^\beta}\right)(0,
0)\right|\le C (g-1)^{1/3}e^{-\tau_j/3}.
\]
The above estimates then enable us to obtain an effective estimate for the Bergman metric. We leave the
detail to the interested reader.

\section{Compact complex manifolds}\label{sec:compact}

There have been extensive studies on the theory of $L^2$-Betti numbers (cf. \cite{CheegerGromov85a,
CheegerGromov85b, Luck94, Yeung94}). In this section, we establish a link between this theory
and Bergman stability on a tower of coverings on a compact complex manifold.

\begin{proposition}\label{prop:link} A tower of coverings $M_j$ on a compact complex manifold $M$
 is Bergman stable if and only if Kazhdan's inequality \eqref{eq:kaz} becomes an equality:
\begin{equation}\label{eq:kaz-eq}
\lim_{j\to\infty}\frac{h^{n, 0}(M_j)}{[\Gamma:\Gamma_j]}= \int_{M} K_{\widetilde M}.
\end{equation}
\end{proposition}

This proposition is a consequence of Proposition~\ref{prop:upper} and the following lemma.

\begin{lemma}\label{lm:equicts} Let $M_j$ be a tower of coverings on a complex manifold. Then the
pull-backs $\widetilde{p}^*_j(K_{M_j})$ of the Bergman kernels  is locally equicontinuous on
$\widetilde{M}$.
\end{lemma}

\begin{proof}  Let $z_0\in \widetilde M$. Let $U\subset\subset B_j=B(z_0, \tau_j(z_0))$ be a neighborhood of
$z_0$ contained in a local coordinate chart. Let $K^*_{M_j}(z, w)$ be Bergman kernel function, representing
of the pull-backs to $\widetilde M$ of the Bergman kernel form $K_{M_j}(z, w)$ on $M_j$. Since
\[
\left\|\widetilde{p}^*_j(K_{M_j})(\cdot, w)\right\|^2_{U}=\left\|K_{M_j}(\cdot,
w)\right\|^2_{\widetilde{p}_j(U)}\le \left\|K_{M_j}(\cdot, w)\right\|^2_{M_j}=\widetilde{p}_j^*(K_{M_j})(w,
w)\le K_{B_j}(w, w)
\]
and $K_{B_j}(w, w)$ converges uniformly on $U$ to $K_{\widetilde M}(w, w)$, the above expressions are
uniformly bounded on $U$.  As a consequence,
\[
\int_{U}\int_{U} |K^*_{M_j}(z, w)|^2\, dV(z)\, dV(w)\le C<\infty.
\]
The equicontinuity of $\widetilde{p}^*_j(K_{M_j})$ near $z_0$ then follows from the Cauchy
estimate.\end{proof}

We now prove Proposition~\ref{prop:link}. The necessity is trivial, following from the uniform convergence
theorem as in the proof of Proposition~\ref{prop:kaz}. To prove the sufficiency, we note that from
Proposition~\ref{prop:upper} and \eqref{eq:kaz-eq}, we have
\[
\limsup_{j\rightarrow \infty} \widetilde{p}_j^*(K_{M_j})(z)= K_{\widetilde{M}}(z).
\]
Thus it suffices to show $\liminf_{j\rightarrow \infty} \widetilde{p}_j^*(K_{M_j})(z)\ge K_{\widetilde{M}}(z)$.
Proving by contradiction, we assume that there exist $z_0\in\widetilde M$ and $\epsilon>0$ such
\[
K^*_{M_{j_k}}(z_0)<K^*_{\widetilde M}(z_0)-\epsilon
\]
for a subsequence $j_k\to\infty$. As before, $K^*$ denotes the function representing the (pull-backs) of
the Bergman kernel forms on a local coordinate chart $U$ near $z_0$. By Lemma~\ref{lm:equicts}, after
possible shrinking of $U$, we have
\[
\widetilde{p}^*_{j_k} (K_{M_j})(z)<K_{\widetilde M}(z)-\frac12\epsilon
\]
for $z\in U$.  It then follows that
\begin{align*}
\limsup_{j_k\to\infty}\frac{h^{n, 0}(M_{j_k})}{[\Gamma:\Gamma_{j_k}]}&=\limsup_{j_k\to\infty} \int_M
K_{M_{j_k}}\le \limsup_{j_k\to\infty}\int_{M\setminus U} K_{M_{j_k}} +\int_U K_{\widetilde
M}-\frac12\epsilon {\rm vol}(U)\\ &\le \int_{M} K_{\widetilde M}-\frac12\epsilon {\rm vol}(U),
\end{align*}
contradicting \eqref{eq:kaz-eq}. We thus conclude the proof of Proposition~\ref{prop:link}.

We now recall relevant facts about the $L^2$-Betti numbers. (We refer the reader to \cite{Atiyah76},
\cite{CheegerGromov85a, CheegerGromov85b, CheegerGromov86}, and \cite[Section~8]{Gromov93} for extensive
discussions on related topics.) Let $\widetilde M$ be a universal covering and let $M_j=\widetilde{M}/\Gamma_j$
be a tower of coverings on a complete Riemannian manifold $M$. Let ${\mathcal H}^s_{(2)}(\widetilde{M})$ be
the space of $L^2$-harmonic $s$-forms on $\widetilde{M}$ corresponding to the $d$-Laplacian $\Delta$. Let
$K^s_{\widetilde{M}}$ be the Schwartz kernel of ${\mathcal H}^s_{(2)}(\widetilde{M})$. The {\it $L^2$-Betti
number} of $M$ is then given by
$$
b^s_{(2)}(M):=\int_M |K^s_{\widetilde{M}}|dV.
$$
When $M$ has bounded geometry and finite volume, Cheeger and Gromov showed that
\begin{equation}\label{eq:b1}
\lim_{j\to\infty}\frac{b^s(M_j)}{[\Gamma:\Gamma_j]} = b^s_{(2)}(M),
\end{equation}
where $b^s(M_j)$ is the ordinary $s$-th Betti number of $M_j$ (\cite{CheegerGromov85a, CheegerGromov85b}).
Similar result was obtained by Yeung \cite{Yeung94} on compact K\"{a}hler manifolds with negative sectional
curvatures. An analogous result was established for a finite connected $CW$-complex by
L\"{u}ck~\cite{Luck94}.

When $M$ is a compact K\"{a}hler manifold, the $L^2$-Hodge number $h^{p,q}_{(2)}(M)$ of $M$ is similarly
given by
\[
h^{p, q}_{(2)}(M):=\int_M |K^{p, q}_{\widetilde{M}}|dV.
\]
where  $K^{p, q}_{\widetilde M}$ is the Schwartz kernel for ${\mathcal H}^{p,q}_{(2)}(\widetilde{M})$, the
space of $L^2-$harmonic $(p,q)-$forms corresponding to the $\bar{\partial}-$Laplacian $\Box$.

\begin{proposition}\label{prop:link2}  A tower of coverings $M_j$ on a compact K\"{a}hler manifold
is Bergman stable if \eqref{eq:b1} holds.
\end{proposition}

\begin{proof} The Hodge-Kodaira decomposition
$$
{\mathcal H}^s_{(2)}(\widetilde{M})=\bigoplus_{p+q=s}{\mathcal H}^{p,q}_{(2)}(\widetilde{M})
$$
implies that
\begin{equation}\label{eq:b2}
b^s_{(2)}(M)=\sum_{p+q=s}h^{p,q}_{(2)}(M).
\end{equation}
By Kazhdan's inequality,
\begin{equation}\label{eq:b3}
\lim\sup_{j\rightarrow \infty} \frac{h^{p,q}(M_j)}{[\Gamma:\Gamma_j]}\le h^{p,q}_{(2)}(M)
\end{equation}
(\cite[Theorem 1 and its proof]{Kazhdan83}, \cite[pp.~153]{Gromov93};  see also Proposition~\ref{prop:kaz}
above), where $h^{p,q}(M_j)$ denotes the ordinary Hodge numbers of $M_j$. Since $b^s(M_j)=\sum_{p+q=s}
h^{p,q}(M_j)$, it follows from \eqref{eq:b1}--\eqref{eq:b3} that
$$
\lim_{j\rightarrow \infty} \frac{h^{p,q}(M_j)}{[\Gamma:\Gamma_j]}= h^{p,q}_{(2)}(M).
$$
In particular, we have \eqref{eq:kaz-eq} and thus the Bergman stability by Proposition~\ref{prop:link}.
\end{proof}

\section{Hyperbolic Riemann surfaces}\label{sec:hyp}

Let $\Gamma$ be a Fuchsian group, i.e., a properly discontinuous subgroup of $SL(2,{\mathbb R})$. Recall
that $\Gamma$ is of {\it convergence type} if $ \sum_{\gamma\in \Gamma}(1-|\gamma (0)|)<\infty. $ We refer the reader to \cite[Chapter~XI]{Tsuji59} for a treatment of the subject. Let $\widetilde{p}$ be the natural projection from $\D$ onto $\D/\Gamma$ (we will also use $\widetilde{p}$ to denote the natural projection from $\D\times \D$ onto $(\D/\Gamma)\times (\D/\Gamma)$). A classical result of Myrberg states that $\Gamma$ is of convergence type if and only if $\D/\Gamma$ is a hyperbolic Riemann surface, and in this case,
\begin{equation}\label{eq:myrberg}
g_{\D/\Gamma}(\widetilde{p}(z),\widetilde{p}(w))=\sum_{\gamma\in \Gamma} g_\Delta(z,\gamma (w))=-\sum_{\gamma\in
\Gamma}\log\left|\frac{z-\gamma(w)}{1-\overline{\gamma(w)}z}\right|
\end{equation}
(see \cite[Theorem~XI.~13]{Tsuji59}).
Note that since $g_\D(z,0)=g_\D(\gamma(z),\gamma(0))$ for any $\gamma\in \Gamma$,
$$
1-|z|^2=\frac{(1-|\gamma(0)|^2)(1-|\gamma(z)|^2)}{|1-\overline{\gamma(0)}\gamma(z)|^2}\le
\frac{4(1-|\gamma(0)|)(1-|\gamma(z))|}{\max\{(1-|\gamma(z)|)^2,(1-|\gamma(0)|)^2\}}.
$$
Therefore,
$$
\frac{(1-|z|^2)(1-|\gamma(0)|)}4\le 1-|\gamma(z)|\le \frac{4(1-|\gamma(0)|)}{1-|z|^2},
$$
from which it follows that when $\Gamma$ is of convergent type, $\sum_{\gamma\in
\Gamma}(1-|\gamma(z)|)<\infty$  all $z\in \D$.

Using
\[
1-\left|\frac{z-\gamma(w)}{1-\overline{\gamma(w)}z}\right|^2=\frac{(1-|z|^2)(1-|\gamma(w)|^2)}
{|1-\overline{\gamma(w)}z|^2}\le 2\frac{1+|z|}{1-|z|} (1-|\gamma(w)|),
\]
and the simple inequality $-\log x\le 2(1-x)$ when $x\ge 1/2$, we then have

\[
-\frac12\log\left|\frac{z-\gamma(w)}{1-\overline{\gamma(w)}z}\right|^2\le
1-\left|\frac{z-\gamma(w)}{1-\overline{\gamma(w)}z}\right|^2\le 2\frac{1+|z|}{1-|z|} (1-|\gamma(w)|),
\]
if $1-|\gamma(w)|\le (1-|z|)/4(1+|z|)$.  Therefore, the series on the right hand side of \eqref{eq:myrberg}
converges local uniformly in $z$ and likewise in $w$.

We now establish a transformation formula of Bergman kernel for a normal covering map between
hyperbolic Riemann surfaces. A related formula for Reinhardt domains in $\C^n$ was obtain in \cite{Fu01}.

\begin{proposition}\label{prop:trans} Let $M$ and $\widetilde M$ be hyperbolic Riemann surfaces.  Let $\widetilde p\colon \widetilde M\to M$ be a normal covering map.  Then
\begin{equation}\label{eq:trans}
\big(\widetilde{p}^*K_{M}\big)(z, w)=\sum_{\gamma\in\Gamma} \big(\gamma^*K^z_{\widetilde M}\big)(w).
\end{equation}
\end{proposition}

\begin{proof} We first prove the case when $\widetilde M=\D$ and $M=\D/\Gamma$ where $\Gamma$ is
a Fuchsian group of convergence type.  Write $K_{\D/\Gamma}(z, w)=K^*_{\D/\Gamma} (z, w) (\frac{i}2 dz\wedge d\bar w)$ where $z$ and $w$ are holomorpohic coordinates induced by the covering map.  Then\eqref{eq:trans} becomes:
\begin{equation}\label{eq:bergman}
K^*_{\D/\Gamma}(\widetilde p (z), \widetilde{p}(w))\widetilde{p}(z)\overline{\widetilde{p}'(w)}=\sum_{\gamma\in\Gamma} K^*_{\D}(z,
\gamma(w))\overline{\gamma'(w)}.
\end{equation}
The above formula then follows from  differentiating both sides of \eqref{eq:myrberg} with respect to $z$ and $\bar w$ and by applying Schiffer's formula \cite{Schiffer46}.

We now prove the general case. Let $\widetilde M=\D/\widetilde\Gamma$ and $M=\D/\Gamma$ with $\widetilde\Gamma$ a normal subgroup of $\Gamma$. Applying \eqref{eq:bergman} to both $\D/\widetilde\Gamma$ and $\D/\Gamma$ and then combining the results, we then obtain formula~\eqref{eq:trans}.
\end{proof}

We are now in position to prove Theorem~\ref{th:hyp}.  Let $M_j=\D/\Gamma_j$ be a tower of coverings on a hyperbolic Riemann surface $M=\D/\Gamma$ and $\widetilde M=\D/\widetilde\Gamma$ be
the top manifold. Then $\Gamma_{j}$ is a decreasing sequence of normal subgroups such that  $\cap\Gamma_j=\widetilde\Gamma$. Let $\widehat\Gamma_j=\Gamma_j/\widetilde\Gamma$.
Let $\widetilde{p}_j\colon \widetilde M\to M_j$ be the natural projection. Applying \eqref{eq:bergman} to $\widetilde{p}_j$, we have
\begin{equation}\label{eq:berg1}
K^*_{M_j}(\widetilde{p}(z), \widetilde{p}(w))\widetilde{p_j}(z)\overline{\widetilde{p}_j^\prime(w)}=\sum_{\widehat\gamma\in\widehat\Gamma_j} K^*_{\widetilde M}(z,
\widehat\gamma(w))\overline{\widehat\gamma'(w)}.
\end{equation}
For $\widehat\gamma\in\widehat\Gamma_j$, write $\widehat\gamma=[\gamma_j]$ with $\gamma_j\in\Gamma_j$. Let $\widehat p\colon \D\to \D/\widetilde\Gamma$ be the natural projection. Let $w\in\D$ and $\widetilde \gamma\in\widetilde\Gamma$. We have
\[
\widetilde\gamma'(\gamma_j(w))\widehat\gamma'(\widehat p(w))=\frac{\widehat p'(\gamma_j(w))\widehat\gamma'(\widehat p(w))}{\widehat p'(\widetilde\gamma(\gamma_j(w)))}=\frac{\widehat p'(\gamma_j(w))(\widetilde\gamma\circ\gamma_j)'(w)}{\widehat p'(w)}.
\]
Therefore, for $z, w\in\D$,
\begin{align}\label{eq:berg2}
K^*_{\widetilde M}(\widehat p(z), \widehat\gamma (\widehat p(w)))\overline{\widehat\gamma'(w)}&=K^*_{\widetilde M}(\widehat p(z), \widehat p(\gamma_j(w)))\overline{\widehat\gamma'(\widehat p(w))} \notag\\
&=\frac{1}{\widehat p'(z)}\frac{1}{\overline{\widehat p'(\gamma_j(w))}}\sum_{\widetilde\gamma\in\widetilde\Gamma} K^*_{\D}(z, \widetilde\gamma(\gamma_j(w)))\overline{\widetilde\gamma'(\gamma_j(w))}\cdot\overline{\widehat\gamma'(\widehat p(w))}\notag\\
&=\frac{1}{\widehat p'(z)}\frac{1}{\overline{\widehat p'(w)}}\sum_{\widetilde\gamma\in\widetilde\Gamma} K^*_{\D}(z, \widetilde\gamma(\gamma_j(w)))\overline{(\widetilde\gamma\circ\gamma_j)'(w)}.
\end{align}
Set
\[
E_j=K^*_{M_j}(\widetilde p(\widehat p(z)), \widetilde p(\widehat p(w)))\widetilde{p}'_j(\widehat p (z))\overline{\widetilde{p}_j^\prime(\widehat p(w))}-K^*_{\widetilde M}(\widehat p(z),
\widehat p(w)).
\]
Combining \eqref{eq:berg1} and \eqref{eq:berg2}, we then have:
\begin{align*}
E_j&=\sum_{\widehat\gamma\in\widehat\Gamma_j\setminus\{1\}} K^*_{\widetilde M}(\widehat p(z),
\widehat\gamma(\widehat p(w)))\overline{\widehat\gamma'(\widehat p(w))}\\
&=\frac{1}{\widehat p'(z)\overline{\widehat p'(w)}}\sum_{[\gamma_j]\in\widehat\Gamma_j\setminus\{1\}}\sum_{\widetilde\gamma\in\widetilde\Gamma}K^*_{\D}(z, \widetilde\gamma\circ\gamma_j (w))\overline{(\widetilde\gamma\circ\gamma_j)'(w)}\\
&=\frac{1}{\widehat p'(z)\overline{\widehat p'(w)}} \sum_{\gamma\in\Gamma_j\setminus\widetilde\Gamma}K^*_{\D}(z, \gamma(w))\overline{\gamma'(w)}.
\end{align*}
It follows from a simple computation that
\[
|E_j|\le \frac{1}{\pi|\widehat p'(z)\widehat p'(w)|(1-|z|)^2(1-|w|)^2}\sum_{\gamma\in\Gamma_j\setminus\widetilde\Gamma} (1-|\gamma(0)|^2).
\]
Since $\cap\Gamma_j=\widetilde\Gamma$, we have $|E_j|\to 0$ locally uniformly as $j\to\infty$. This concludes the proof of Theorem~\ref{th:hyp}.

\begin{remark} We do not use the condition $[\Gamma: \Gamma_j]<\infty$ in the above proof. Theorem~\ref{th:hyp} remains true even if the finiteness assumption on the indices $[\Gamma:\Gamma_j]$ is dropped from the definition of tower of coverings.
\end{remark}

\section{Complete K\"{a}hler manifolds}\label{sec:g}

In this section, we study towers of coverings on complete K\"{a}hler manifolds. We first establish
auxiliary spectral theoretic results. Let $(M, \omega)$ be a complete K\"{a}hler manifold.  Let
$\square^{M}_{p, q}$ be the $\dbar$-Laplacian on $L_{(2)}^{p, q} (M)$. We will use  $\sigma(\square)$ and
$\sigma_e(\square)$ to denote the spectrum and the essential spectrum of $\square$ respectively.  The
following lemma is well-known; we provide a proof for completeness.

\begin{lemma}\label{lm:ess} Suppose there is a compact set $K\subset M$ and a constant
$C>0$ such that
\begin{equation}\label{eq:f}
\|u\|^2\le C\left(\|\bar{\partial} u \|^2+\|\bar{\partial}^\ast u\|^2+\int_K |u|^2 dV\right)
\end{equation}
holds for all $u\in {\rm Dom\,}\bar{\partial} \cap {\rm Dom\,}\bar{\partial}^\ast\cap L^{p,q}_{(2)}(M)$.
Then $\sigma_e(\square^{p, q})\subset [\frac1C,\ 0)$.
\end{lemma}

\begin{proof} Let $\lambda\in\sigma_e(\square)$. Then by the Weyl criteria (cf.
\cite[Theorem~7.24]{Weidmann80}), there exists a sequence $u_j\in\dom(\square)$ such that
$\|u_j\|=1$, $\|(\square-\lambda)u_j\|\to 0$, and $u_j\to 0$ weakly. Thus
\[
\lim_{j\to\infty} \left(\|\dbar u_j\|^2+\|\dbar^* u_j\|^2\right)=\lambda.
\]
By the interior ellipticity of $\square$ and the Rellich compactness theorem, there exists a subsequence
$u_{j_k}$ converging in the $L^2$-norm on $K$. By the assumption, the limit must be 0.  Plugging $u_{j_k}$
into \eqref{eq:f} and taking the limit, we then obtain $\lambda\ge 1/C$. \end{proof}

Note that if $K$ is an empty set, then \eqref{eq:f} is equivalent to $\sigma(\square)\subset [\frac1C, \
\infty)$.  Furthermore, if $\sigma_{e}(\square)\subset [\frac1C,\ \infty)$, then $\sigma(\square)\cap [0,
\ \frac1C)$ is either empty or consists of eigenvalues of finite multiplicities (cf. \cite[Theorem
4.5.2]{Davis95}).

Now let $(M,\omega)$ be as in Theorem~\ref{th:g}. Then there exists a $C^2$ psh function $\psi$ on $M\setminus K$
satisfying
$$
C_0^{-1}\omega \le \partial\bar{\partial}\psi \le C_0 \omega,\ \ \ |\bar{\partial}\psi|^2_\omega\le C_0
$$
where $K$ is a compact subset of $M$ and $C_0>0$ is a constant. After a multiple of a cut-off function, we
may assume that $\psi$ is a $C^2$ real-valued function on $M$ such that the above inequalities hold outside
a geodesic ball $B(z_0,R)=\{z\in M; \ d_M(z_0, z)<R\}$ where $d_M(z_0,\cdot)$ is the distance to a fixed point
$z_0\in M$. Write $\psi_j=p_j^\ast(\psi)$, $\omega_j^\ast=p_j^\ast(\omega)$, and
$d_j (\cdot)=p_j^\ast(d_M(z_0,\cdot))$ where $p_j:M_j\rightarrow M$ is the natural projection.  Note that $d_j$ need not be the distance function of $M_j$. Let $K_j=p^{-1}_j(\overline{B(z_0,2R)})$.

\medskip

\begin{lemma}\label{lm:f2}  There is a
constant $C_1=C(n,C_0,R)>0$ such that
\begin{equation}\label{eq:fundamental}
\|u\|^2\le C_1\left(\|\bar{\partial} u \|^2+\|\bar{\partial}^\ast u\|^2+\int_{K_j} |u|^2 dV\right)
\end{equation}
holds for all $u\in {\rm Dom\,}\bar{\partial} \cap {\rm Dom\,}\bar{\partial}^\ast\cap L_{n,1}^{2}(M_j)$.
As a consequence, $\sigma_e(\square_{n, 1}^{M_j})\subset [\frac1{C_1}, \infty)$.
\end{lemma}

\begin{proof} The method, which goes back to Donnelly-Fefferman \cite{DonnellyFefferman83} and Gromov
\cite{Gromov91}, is well known (see \cite{OhsawaTakegoshi87, DiederichOhsawa94,Siu96, BerndtssonCharpentier00, Mcneal02, ChenFu11} for related results). We provide a proof for completeness. Let  $v\in C_0^{n,1}(M_j)$
such that $v=0$ on $p_j^{-1}(B(z_0,R))$. By the Bochner-Kodaira-Nakano formula, we have for any
$\tau>0$,
\begin{equation}\label{eq:add1}
\|\bar{\partial} {v}\|^2_{\tau\psi_j}+\|\bar{\partial}^\ast_{\tau\psi_j} v\|^2_{\tau\psi_j}\ge \int_{{M}_j}
\langle [\sqrt{-1}\tau\partial\bar{\partial}\psi_j,\Lambda]v,v\rangle e^{-\tau\psi_j}dV\ge
C_0^{-1}\tau\|v\|^2_{\tau\psi_j}
\end{equation}
where $\Lambda$ is the adjoint of $Lu=\omega\wedge u$,  $\|\cdot\|_{\tau\psi_j}$ the $L^2-$norm with weight $\tau\psi_j$, and
$\bar{\partial}^\ast_{\tau\psi_j}$ the adjoint of $\bar{\partial}$ with respect to the inner product
$\langle\cdot,\cdot\rangle_{\tau\psi_j}$ (see \cite[p.~68]{B02}). Let $w=e^{-\tau\psi_j/2}v$. Note that
\begin{align*}
\bar{\partial} v & = e^{\tau\psi_j/2}\left(\bar{\partial}w+\frac{\tau}2\bar{\partial}\psi_j\wedge
w\right)\\
\intertext{and}
\bar{\partial}^\ast_{\tau\psi_j} v & =  e^{\tau\psi_j/2}\left(\bar{\partial}^\ast
w-\frac{\tau}2\bar{\partial}\psi_j\lrcorner\, w\right)
\end{align*}
where $"\lrcorner"$ is the contraction operator. It follows Schwarz' inequality that
$$
\|\bar{\partial} {v}\|^2_{\tau\psi_j}+\|\bar{\partial}^\ast_{\tau\psi_j} v\|^2_{\tau\psi_j}\le
2\|\bar{\partial}w\|^2+2\|\bar{\partial}^\ast w\|^2+C_0\tau^2\|w\|^2.
$$
Substituting this inequality into \eqref{eq:add1}, we have
\begin{equation}\label{eq:add2}
\|\bar{\partial}w\|^2+\|\bar{\partial}^\ast w\|^2\ge \frac{\tau}2 (C_0^{-1}-C_0\tau)\|w\|^2\ge C_1\|w\|^2
\end{equation}
provided $\tau<C_0^{-2}/2$. Now fix such a $\tau$.
 Let $0\le \chi\le 1$ be a $C^\infty$ cut-off function
such that $\chi=0$ on $(-\infty,1)$ and $\chi=1$ on $(2,\infty)$. Let $u\in
C_0^{n,1}(M_j)$ and $w=\chi(d_j/R) u$. Then
\begin{eqnarray*}
\bar{\partial}w & = & \chi(d_j/R)\bar{\partial} u+ \bar{\partial}\chi(d_j/R)\wedge u; \\ \bar{\partial}^\ast
w & = & \chi(d_j/R)\bar{\partial}^\ast u- \bar{\partial}\chi(d_j/R)\lrcorner\,u.
\end{eqnarray*}
From \eqref{eq:add2} and Schwarz's inequality, we have
$$
\|u\|^2\le C_1\left(\|\bar{\partial} u \|^2+\|\bar{\partial}^\ast u\|^2+\int_{K_j} |u|^2 dV\right).
$$
By Andreotti-Vesentini's approximation theorem \cite{AndreottiVesentini65}, the same inequality holds for
all $u\in {\rm Dom\,}\bar{\partial} \cap {\rm Dom\,}\bar{\partial}^\ast\cap L_{n,1}^{2}(M_j)$.  From
Lemma~\ref{eq:f}, we have $\sigma_e(\square_{n, 1}^{M_j})\subset [\frac1{C_1}, \infty)$.\end{proof}

Let $C_1$ be the constant in Lemma~\ref{lm:f2}. For $0<\delta<\frac1{C_1}$, let ${\mathcal
H}^{n,1}_{(2)}(M_j, \delta)$ be the linear span of $(n,1)$ eigenforms of $\Box$, with corresponding
eigenvalues smaller than or equal to $\delta$. It is a finite dimensional complex vector space. Let
$\{\phi_k\}$ be a orthonomal basis of eigenforms in ${\mathcal H}^{n,1}_{(2)}(M_j, \delta)$, we define the
corresponding Bergman kernel function as
$$
|K^1_{M_j,\delta}|=\sum |\phi_k|^2.
$$
Then we have
$$
{\rm dim\,}{\mathcal H}^{n,1}_{(2)}(M_j, \delta)=\int_{M_j} |K^1_{M_j,\delta}|dV.
$$
It is easy to see that
\begin{equation}\label{eq:bern}
|K^1_{M_j,\delta}|(z)\le n \sup\left\{|f|^2(z):f\in {\mathcal H}^{n,1}_{(2)}(M_j, \delta), \|f\|=1\right\}
\end{equation}
(e.g., \cite[Lemma 4.1]{Berndtsson02}).

\begin{lemma}\label{lm:f3} For every $\epsilon>0$, there exist
$0<\delta_0<\frac1{C_1}$ and $j_0>0$ such that for $\delta\le \delta_0$ and $j\ge j_0$,
$$
|K^1_{M_j,\delta}|(z)<\epsilon, \ \ \ \forall z\in K_j.
$$
\end{lemma}

\begin{proof} Let $z\in K_j$. Let $f\in {\mathcal H}^{n, 1}(M_j, \delta)$ be the form that realizes
the supremum on the right side of \eqref{eq:bern}.  Let $\kappa$ be a $C^\infty$ cut-off function such that
 $0\le\kappa\le 1$, $\kappa=1$ on $(-\infty, 1/2)$, and
$\kappa=0$ on $(1,\infty)$. Let
$$
\rho=\kappa(d_{M_j}(z,\cdot)/\tau_j(z))f.
$$
Here we use $\tau_j(\cdot)$ to denote also the push-down to $M_j$ of $\tau_j(\cdot)$ from $\widetilde M$ (as defined by \eqref{eq:tau} with $\Gamma$ replaced by $\Gamma_j$). Recall that
$\widetilde{p}_j\colon\widetilde{M}\to M_j$ is the natural projection. Let $\widetilde{\rho}=\widetilde{p}_j^*(\rho)$. By a similar argument as in the proof of Lemma~\ref{lm:f2}, we have
\begin{eqnarray*}
\|\widetilde{\rho}\|^2_{\widetilde{M}} &\le & C
(\|\bar{\partial}\widetilde\rho\|^2_{\widetilde{M}}+\|\bar{\partial}^\ast
\widetilde\rho\|^2_{\widetilde{M}})\\ &\le&
C\left(\frac{\sup|\kappa'|^2}{\tau_j^2(z)}+\|\bar{\partial}f\|^2_{M_j}+\|\bar{\partial}^\ast
f\|^2_{M_j}\right)\\ &\le& C\left(\frac{\sup|\kappa'|^2}{\tau_j^2(z)}+\delta\right).
\end{eqnarray*}
Since $p_j(K_j)=\overline{B(z_0,2R)}$ is compact, there is a constant $r=r(R)>0$ such that
$\tau_j(z)>2r$ for all $j$. Thus $ \rho= f$ on $B(z,r)$. Using G{\rm \aa}rding's inequality  together with Sobolev's estimates, we have
that for sufficiently large $j$ and $m>n$,
\begin{eqnarray*}
|f|^2(z) &\le & C\left(\int_{B(z,r)}|f|^2dV+\int_{B(z,r)}|\Box^{(m)} f|^2dV\right)\\ & \le & C
\left(\|\widetilde\rho\|^2_{\widetilde{M}}+\|\Box^{(m)} f\|^2_{M_j}\right)\\ &\le & C
\left(\frac{\sup|\kappa'|^2}{\tau_j^2(z)}+\delta+\delta^{2m}\right),
\end{eqnarray*}
where the constants depending only on $m$ and $r$.  Lemma~\ref{lm:f3} then follows from
\eqref{eq:bern} and Lemma~\ref{lm:inj}. \end{proof}

\begin{lemma}\label{lm:f4} For every $\epsilon>0$, there exist
$0<\delta_0<\frac1{C_1}$ and $j_0>0$ such that for each $\delta\le \delta_0$ and $j\ge j_0$,
$$
{\rm dim\,}{\mathcal H}^{n,1}_{(2)}(M_j, \delta)\le \epsilon\,[\Gamma:\Gamma_j]\, {\rm vol}(B(z_0,2R)).
$$
\end{lemma}

\begin{proof} Let $\{\phi_k\}$ be an orthonormal basis of ${\mathcal H}^{n,1}_{(2)}(M_j, \delta)$ as above.
The estimate \eqref{eq:fundamental} implies
$$
1=\|\phi_k\|^2\le C\left(\delta+\int_{K_j} |\phi_k|^2 dV\right).
$$
Summing up, we get
\begin{align*}
{\rm dim\,}{\mathcal H}^{n,1}_{(2)}(M_j, \delta)&\le \frac{C}{1-C\delta}\int_{K_j} |K^1_{M_j,\delta}|^2
dV<\frac{C \epsilon}{1-C\delta} {\rm vol}(K_j)\\ &=\frac{C\epsilon}{1-C\delta}[\Gamma:\Gamma_j]  {\rm
vol}(B(z_0,2R)).
\end{align*}
\end{proof}

Observe that every positive eigenvalue $\lambda$ of $\Box$ on $L^{n,0}_{(2)}(M_j)$ is also a
eigenvalue of $\Box$ on $L^{n,1}_{(2)}(M_j)$: If $\{f\}$ is a normalized eigenform of $\Box$ on $L^{n, 0}_{(2)}(M_j)$
associated with $\lambda$, then
$$
\square(\bar\partial f)=\lambda \bar\partial f \quad \text{and}\quad \|\bar{\partial}f\|^2_{M_j}=(\Box f, f)=\lambda.
$$
Thus $\bar{\partial}$ induces a linear injection of ${\mathcal H}^{n,0}_{(2)}(M_j,\delta)\ominus{\mathcal
H}^{n,0}_{(2)}(M_j)$ to ${\mathcal H}^{n,1}_{(2)}(M_j,\delta)$, and
\begin{equation}\label{eq:f2}
{\rm dim\,}\left({\mathcal H}^{n,0}_{(2)}(M_j,\delta)\ominus{\mathcal H}^{n,0}_{(2)}(M_j)\right) \le {\rm
dim\,} {\mathcal H}^{n,1}_{(2)}(M_j,\delta)<\infty
\end{equation}
where ${\mathcal H}^{n,0}_{(2)}(M_j,\delta)$ is the linear span of $(n,0)$ eigenforms of $\Box$, with
corresponding eigenvalues smaller than or equal to $\delta$. Let $|K_{M_j,\delta}|$ and $|K_{M_j}|$ be the
Bergman kernel functions of ${\mathcal H}^{n,0}_{(2)}(M_j,\delta)$ and ${\mathcal H}^{n,0}_{(2)}(M_j)$
respectively.

\begin{lemma}\label{lm:f5} For every $\epsilon>0$, there exists
$0<\delta_0<\frac1{C_1}$ and $j_0>0$ such that for $\delta\le \delta_0$ and $j\ge j_0$,
$$
0<|K_{M_j,\delta}|(z)-|K_{M_j}|(z)<\epsilon,\ \ \ \forall\, z\in K_j.
$$
\end{lemma}

\begin{proof} By \eqref{eq:f2} and Lemma~\ref{lm:f4}, we have for $\delta\le
\delta_0$ and $j\ge j_0$,
$$
\int_{M_j}(|K_{M_j,\delta}|-|K_{M_j}|) dV\le \epsilon\,[\Gamma:\Gamma_j]\, {\rm vol}(B(z_0,2R)).
$$
 Since the function in the integral is invariant under
$\Gamma/\Gamma_j$,
\begin{equation}\label{eq:f3}
\int_{M} (|K_{M_j,\delta}|-|K_{M_j}|) dV< \epsilon\,{\rm vol}(B(z_0,2R)).
\end{equation}
Now take a normal coordinate ball $B(z,\epsilon_0)$ around $z$ (lifted from $M$) where $\epsilon_0$ depends
only on $M$. Again it follows from G{\rm \aa}rding's inequality and Sobolev's estimates that for $m>n$,
\begin{align}
|f|^2(z)&\le C_{m,\epsilon_0} \left( \int_{B(z,\epsilon_0)} |f|^2dV+ \int_{B(z,\epsilon_0)}|\Box^{(m)}
f|^2dV\right)\notag\\ &\le C_{m,\epsilon_0}(1+\delta^{2m})\int_{M} |f|^2dV \label{eq:f4}
\end{align}
for all $f\in {\mathcal H}^{n,0}_{(2)}(M_j,\delta)\ominus{\mathcal H}^{n,0}_{(2)}(M_j)$. By \eqref{eq:f3}
and \eqref{eq:f4},
$$
|K_{M_j,\delta}|(z)-|K_{M_j}|(z)\le
 C_{m,\epsilon_0}(1+\delta^{2m})\epsilon\,{\rm
vol}(B(z_0,2R)),
$$
completing the proof. \end{proof}

\medskip

We now proceed to prove Theorem~\ref{th:g}. By Proposition~\ref{prop:upper}, it suffices to verify the lower semicontinuity of the Bergman kernels $K_{M_j}$ as $j\to\infty$.  Let $z\in\widetilde M$.  Let $R$ be sufficiently large so that $\widetilde p_1(z)\in B(z_0,R)$. Let $\delta_0$ and $j_0$ be chosen as in Lemma~\ref{lm:f5}. Let $f$ be a candidate for the extremal property \eqref{eq:extreme} of $K_{\widetilde{M}}(z)$ and let
$$
\varrho_j=\kappa(d_{\widetilde{M}}(z,\cdot)/\tau_j(z))f
$$
where $\kappa$ is the cut-off function as above.  Since $\varrho_j$ is supported in a Dirichlet fundamental domain of $M_j$, we may push down $\varrho_j$ onto $M_j$ and regard it as a $(n,0)-$form on
$M_j$ with $\|\varrho_j\|_{M_j}\le 1$. Then we have the orthonomal decomposition
$$
\varrho_j=u_j+v_j
$$
with $u_j\in {\mathcal H}^{n,0}_{(2)}(M_j,\delta_0)$ and
$$
(\Box v_j,v_j) \ge \delta_0 \|v_j\|^2.
$$
By Lemma~\ref{lm:f5}, we have
\begin{equation}\label{eq:f5}
|u_j|^2(z) \le |K_{M_j,\delta_0}|(z)\|u_j\|^2\le |K_{M_j,\delta_0}|(z) < |K_{M_j}|(z)+\epsilon
\end{equation}
for $j\ge j_0$, whilst
\begin{equation}\label{eq:f6}
\|v_j\|^2\le \delta_0^{-1}(\Box v_j, v_j)=\delta_0^{-1}\|\bar{\partial}v_j\|^2.
\end{equation}
Since $\Box u_j\in {\mathcal H}^{n,0}_{(2)}(M_j,\delta_0)$, we have
$$
(\bar{\partial}u_j, \bar{\partial}v_j)=(\Box u_j,v_j)=0.
$$
Hence
\begin{equation}\label{eq:f7}
\|\bar{\partial}v_j\|^2\le \|\bar{\partial}\varrho_j\|^2\le C\sup |\kappa'|^2 \tau_j^{-2}(z).
\end{equation}
By \eqref{eq:f6} and \eqref{eq:f7},
\begin{equation}\label{eq:f8}
\|v_j\|^2\le C\frac{\sup|\kappa'|^2}{\delta_0} \tau_j^{-2}(z).
\end{equation}
In order to obtain a pointwise estimate of $v_j$, we fix a coordinate unit ball ${\mathbb B}^n$ in $\widetilde M$, centered at $z$ and lifted from $M$.  It follows from G{\rm \aa}rding's inequality together with Sobolev's estimates that for
$m>n$,
\begin{eqnarray*}
|\bar{\partial} u_j|^2(z) & \le & C_{m} \left( \int_{{\mathbb B}^n} |\bar{\partial}u_j|^2dV+ \int_{{\mathbb
B}^n}|\Box^{(m)}(\bar{\partial}u_j)|^2dV\right)\\ &\le &  C_{m}(1+\delta_0^{2m})\|\bar{\partial}u_j\|^2\le
C_{m}(1+\delta_0^{2m}) \|\bar{\partial}\varrho_j\|^2
\end{eqnarray*}
for $z\in {\mathbb B}^n_{1/2}$. (Here we identify $u_j$ and $v_j$ with their pull-backs from $M_j$ onto $\widetilde M$.) Since $\varrho_j$ is holomorphic in ${\mathbb B}^n$ for large $j$, we
conclude that
\begin{equation}\label{eq:f9}
|\bar{\partial} v_j|^2=|\bar{\partial} u_j|^2\rightarrow 0
\end{equation}
 uniformly on ${\mathbb B}^n_{1/2}$ as $j\rightarrow \infty$.
 Let
$ K_{\rm BM} $ be the Bochner-Martinelli kernel. Then
$$
v_j(0)= \int_{{\mathbb B}^n} \bar{\partial}(\kappa(2|z|)v_j)\cdot K_{\rm BM}= \int_{{\mathbb B}^n} v_j
\bar{\partial} \kappa(2|z|)\cdot K_{\rm BM}+ \int_{{\mathbb B}^n} \kappa(2|z|) \bar{\partial} v_j \cdot
K_{\rm BM}
$$
converges to $0$ by \eqref{eq:f8} and \eqref{eq:f9}, since $K_{\rm BM}$ is $L^1$ on ${\mathbb B}^n$.
Combing this fact with \eqref{eq:f5}, we conclude that
$$
|K_{\widetilde{M}}|(z)\le {\lim\inf}_{j\rightarrow \infty} |K_{M_j}|(z).
$$
This conclude the proof of Theorem~\ref{th:g}.

\begin{corollary}\label{co:hyp}  Any tower of coverings on a hyperconvex complex manifold is Bergman stable.
\end{corollary}

\begin{proof} Let $M_j=\widetilde M/\Gamma_j$ be a tower of covering on a hyperconvex complex manifold $M$. Let
$\rho\colon M\to [-1, \ 0)$ be a smooth strictly plurisubharmonic proper map. Put
$\psi=-\log(-\rho)$, $\omega=\partial\bar\partial\psi$, $\widetilde\psi=\widetilde{p}^*(\psi)$, and
$\widetilde\omega=\widetilde{p}^*(\omega)$, where $\widetilde{p}\colon \widetilde{M}\to M$ is the natural
projection. Then the assumptions in Theorem~\ref{th:g} are satisfied.
\end{proof}

The following theorem is a simple variation of Theorem~\ref{th:g}:

\begin{theorem}\label{th:g1} A tower of coverings $M_j$ on a complete Hermitian manifold $(M,\omega)$
is Bergman stable if the following conditions hold:
\begin{enumerate}
\item There exist a compact set $K\subset M$, a $C^\infty-$smooth plurisubharmonic function on $M\backslash K$, and a constant $C>0$ such that $\omega=\partial\bar{\partial}\psi$ and $\partial\bar{\partial}\psi\ge C^{-1}\partial\psi\wedge \bar{\partial}\psi$ on $M\backslash K$.
\item There exist a $C^\infty$-smooth plurisubharmonc function $\tilde{\psi}$ on the top manifold $\widetilde{M}$ and a constant $\widetilde C>0$ such that $\partial\bar{\partial}\tilde{\psi}\ge \widetilde{C}^{-1}\tilde{\omega}$ and $\partial\bar{\partial}\tilde{\psi}\ge \widetilde{C}^{-1}\partial\tilde{\psi}\wedge \bar{\partial}\tilde{\psi}$, where $\widetilde{\omega}$ is the lift of $\omega$ to $\widetilde{M}$.
\end{enumerate}
\end{theorem}

We indicate how the proof of Theorem~\ref{th:g} can be easily modified to obtain a proof of Theorem~\ref{th:g1}. Notice that $M$ is K\"{a}hler on $M\setminus K$.  Clearly, the conclusion of Lemma~\ref{lm:f2} is valid under condition $(1)$ above. The proof of Lemma~\ref{lm:f3} can be modified as follows: Since for each $(n,q)-$form on $\widetilde{M}$, the $L^2-$norm with respect to $\partial\bar{\partial}\tilde{\psi}$ is always dominated by the $L^2-$norm with respect to the metric $\tilde{\omega}$. Thus
\begin{align*}
\|\tilde{\rho}\|^2_{\widetilde{M},\partial\bar{\partial}\tilde{\psi}}&\le C\left(\|\bar{\partial}\tilde{\rho}\|^2_{\widetilde{M},\partial\bar{\partial}\tilde{\psi}}+
\|\bar{\partial}^\ast\tilde{\rho}\|^2_{\widetilde{M},\partial\bar{\partial}\tilde{\psi}}\right)\\
&\le C\left(\|\bar{\partial}{\rho}\|^2_{M_j}+\|\bar{\partial}^\ast{\rho}\|^2_{M_j}\right)\le C\left(\frac{\sup|\kappa'|^2}{\tau_j^2(z)}+\delta\right).
\end{align*}
The other parts of the arguments remains unchanged after replacing $\|\tilde{\rho}\|_{\widetilde{M}}$ by $\|\tilde{\rho}\|_{\widetilde{M},\partial\bar{\partial}\tilde{\psi}}$.

We are now in position to prove Theorem~\ref{th:r}.  Let $M$ be a Riemann surface. The case is trivial if the universal covering of $M$ is ${\mathbb P}^1$ since $M$ is also ${\mathbb P}^1$ in this case (\cite[Theorem~IV.6.3, p.193]{FarkasKra80}). If the universal covering of $M$ is ${\mathbb C}$, then $M$ and its normal coverings are conformally equivalent to ${\mathbb C}$, the punctured complex plane ${\mathbb C}^\ast$, or a torus (\cite[Theorem~IV.6.4, p.193]{FarkasKra80}). Thus the Bergman kernels of normal coverings of $M$ vanish identically when $M$ is non-compact and Bergman stability is established in this case. In case when $M$ is a torus,  each covering $M_j$ is the tower is also a torus except the universal covering ${\mathbb C}$. Since holomorphic 1-forms on a torus have the form of $c dz$ where $c$ is a complex number, we have $|K_{M_j}|=1/{\rm vol\,}M_j$, which tends to zero as $j\rightarrow \infty$. Thus we have Bergman stability in this case. It suffices to deal with the case when the universal covering of $M$ is the unit disc ${\mathbb D}$. According to Rhodes' theorem and Theorem~\ref{th:hyp}, it suffices to consider the case when $M$ is parabolic. By a theorem of Nakai~\cite{Nakai62},
there exists a harmonic function $u$ outside a compact subset such that $u(z)\rightarrow +\infty$ as $z$ tends to the ideal boundary $\partial M$ of $M$.  Let $K$ be a compact subset of $M$ such that $u>1$ on $M\backslash K$. By Richberg's theorem, there is a $C^\infty-$smooth strictly subharmonic function $\phi$ on $M\backslash K$ such that $|\phi-(-u)|\le 1$ holds on $M\backslash K$.  Let $0\le \chi\le 1$ be a function in $C^\infty_0(M)$ such that $\chi=1$ on a neighborhood of $K$, and let $0\le \kappa\le 1$ be a function in $C^\infty_0(M)$ such that $\kappa=1$ in a neighborhood of ${\rm supp\,}\chi$.  Let $\omega_{0}$ be a K\"ahler metric on $M$. Put
$$
\omega=C_1\kappa\,\omega_{0}+\partial\bar{\partial}\left((1-\chi)(-\log(-\phi))\right)
$$
where $C_1>0$ is a constant.
We see that $\omega$ is a complete Hermitian metric on $M$ satisfying condition $(1)$ in Theorem~\ref{th:g1} provided $C_1$ is sufficiently large (we may take $\psi=-\log(-\phi)$). Let $\tilde{p}:\D\rightarrow M$ be the natural projection. We define
$$
\widetilde{\psi}(z)=C_2\left(-\log(1-|z|^2)\right)+\tilde{p}^\ast\left((1-\chi)(-\log(-\phi))\right)
$$
where $C_2>0$ is a constant. Since $\partial\bar{\partial}(-\log(1-|z|^2))$ descends to a complete K\"ahler metric on $M$, $\partial\bar{\partial}\tilde{\psi}$ dominates $\tilde{\omega}$ provided $C_2$ sufficiently large. It is also easy to verify $\partial\bar{\partial}\tilde{\psi}\ge C^{-1}\partial \tilde{\psi}\wedge \bar{\partial}\tilde{\psi}$. Applying Theorem~\ref{th:g1}, we then conclude the proof of Theorem~\ref{th:r}.

\section{Appendix: Applications}\label{sec:app}

In this appendix, we provide some applications of Theorem~\ref{th:g} on quotients of polydiscs and balls.

Let ${\mathbb{H}}^n$ be the $n$-product of the upper half planes $\mathbb{H}$ with the Bergman metric. The
connected component $G$ of the identity of the group of all holomorphic automorphisms of ${\mathbb{H}}^n$
contains of all the transformations of the form $\sigma=(\sigma^{(1)},\cdots,\sigma^{(n)})$,
$\sigma^{(i)}\in PSL(2,{\mathbb R})$. An element $\sigma$ of $G$ is {\it parabolic} if each $\sigma^{(i)}$
has exactly a fixed point on $\overline{\mathbb{R}}$. Let $\Gamma$ be a Hilbert modular group. (We refer
the reader to \cite{Tamagawa59, Shimizu63} for the relevant material.) Then $\Gamma$ has a fundamental domain
$F$ of the form
$$
F=F_0\cup V_1\cup\cdots\cup V_t
$$
where $F_0$ is relatively compact in ${\mathbb{H}}^n$, the $V_j$'s are disjoint, each $V_j\subset
\sigma_j^{-1}(U_j)$ is the fundamental domain of the group $\Gamma_j$ of all $\gamma\in \Gamma$ fixing some
point $x_j$ in $\overline{\mathbb{R}}^n$, and
$$
U_j=\{z\in \mathbb{C}^n:{\rm Im\,}z_{1}\times\cdots\times{\rm Im\,}z_{n}>d_j\}
$$
with $d_j$ being a suitably chosen positive number, $\sigma_j$ an element in $G$ such that
$\sigma_j(x_j)=\infty$ (see \cite[p.~48]{Shimizu63}). Since each nontrivial element of $\sigma_j\Gamma_j
\sigma_j^{-1}$ fixes exactly one point $\infty$, it has to be a translation.  Let $\D^n$ be the unit polydisc in ${\mathbb C}^n$ and let $\Gamma$ be a Hilbert modular group.

\begin{proposition} Any tower of coverings $M_j=\D^n/\Gamma_j$ on $M=\D^n/\Gamma$ is
Bergman stable.
\end{proposition}

\begin{proof} Since $\D^n$ is biholomorphic to ${\mathbb{H}}^n$, it suffices to prove the proposition for
${\mathbb{H}}^n$.  The Bergman kernel
$$
K_{\mathbb{H}^n}(z)=\frac1{(4\pi)^n}\frac1{({\rm Im\,}z_1\times\cdots\times{\rm Im\,}z_n)^2}
$$
of $\mathbb{H}^n$ is invariant under translations (in particular $\sigma_j\Gamma_j
\sigma_j^{-1}-$invariant) and a direct computation shows $\partial\bar{\partial}\log
K_{{\mathbb{H}^n}}\gtrsim \partial\log K_{{\mathbb{H}^n}}\wedge\bar{\partial}\log K_{{\mathbb{H}^n}}$. By
setting $\psi=\sigma_j^\ast\log K_{\mathbb{H}^n}$ on each parabolic end $V_j/\Gamma_j$, we see that
conditions in Theorem~\ref{th:g} are satisfied.  \end{proof}

\begin{proposition}\label{prop:ball} An tower of coverings on a complete K\"ahler manifold with pinched
negative sectional curvature and finite volume is Bergman stable.
\end{proposition}

\begin{proof}  Let $(M,\omega)$ be a  complete K\"ahler manifold with pinched negative sectional curvature
and finite volume. Let $(\widetilde{M}, \widetilde\omega)$ be its universal covering. According to Siu-Yau \cite{SiuYau82},
each Buseman function $\psi$ on $\widetilde{M}$ satisfies $C^{-1}\widetilde\omega \le \partial\bar{\partial}\psi\le
C\widetilde\omega$ and $\partial\bar{\partial}\psi\ge C^{-1}\partial\psi\wedge \bar{\partial}\psi$. Furthermore, $M$
is the union of a compact set and a finite number of cusp ends such that each end admits a function which
is a push-down of some Busemann function on $\widetilde{M}$. Thus the conditions of Theorem~\ref{th:g} are satisfied.
\end{proof}

\begin{remark} Proposition~\ref{prop:ball} also holds when $M$ is a ball quotient that is geometrically finite in the sense of Bowditch \cite{Bowditch95}.  We leave the detail to the interested reader.
\end{remark}

\end{document}